\documentclass[12pt,a4]{amsart}
\usepackage[a4paper, left=28mm, right=28mm, top=28mm, bottom=34mm]{geometry}
\usepackage[all]{xy}
\usepackage{amsmath}
\usepackage{amssymb}
\usepackage{amsthm}
\usepackage{mathrsfs}
\theoremstyle{definition}
\newtheorem{thm}{Theorem}[section]
\newtheorem{lem}[thm]{Lemma}
\newtheorem{cor}[thm]{Corollary}
\newtheorem{prop}[thm]{Proposition}

\theoremstyle{definition}

\newtheorem{rem}[thm]{Remark}

\numberwithin{equation}{section}
\def\A{{\mathbb A}}
\def\F{{\mathbb F}}
\def\Q{{\mathbb Q}}
\def\R{{\mathbb R}}
\def\Z{{\mathbb Z}}
\def\C{{\mathbb C}}

\def\Frob{\mathop{\mathrm{Frob}}\nolimits}

\def\Aut{\mathop{\mathrm{Aut}}\nolimits}
\def\Gal{\mathop{\mathrm{Gal}}\nolimits}

\def\NS{\mathop{\mathrm{NS}}\nolimits}
\def\GL{\mathop{\mathrm{GL}}\nolimits}
\def\SL{\mathop{\mathrm{SL}}\nolimits}

\def\Tr{\mathop{\mathrm{Tr}}\nolimits}
\def\det{\mathop{\mathrm{det}}\nolimits}

\begin{document}

\title[Galois representations]
{Galois representations associated with a non-selfdual automorphic
representation of $\GL(3)$}

\author{Tetsushi Ito}
\address{Department of Mathematics, Faculty of Science, Kyoto University, Kyoto 606-8502, Japan}
\email{tetsushi@math.kyoto-u.ac.jp}

\author{Teruhisa Koshikawa}
\address{Research Institute for Mathematical Sciences, Kyoto University, Kyoto 606-8502, Japan}
\email{teruhisa@kurims.kyoto-u.ac.jp}

\author{Yoichi Mieda}
\address{Graduate School of Mathematical Sciences, The University of Tokyo,
3-8-1 Komaba Meguro-ku Tokyo 153-8914, Japan} 
\email{mieda@ms.u-tokyo.ac.jp}

\date{\today}

\subjclass[2010]{Primary 11R39; Secondary 11F80, 11F70}
\keywords{Galois representations, automorphic representations, motives}


\date{\today}

\maketitle

\begin{abstract}
In 1994, van Geemen and Top constructed
a non-selfdual motive of rank three over $\Q$
conjecturally associated with
a cuspidal non-selfdual automorphic representation of $\GL_3(\A_{\Q})$
of level $\Gamma_0(128)$.
They experimentally confirmed the coincidence of the local $L$-factors at finitely many primes using computer.
In this paper, we shall prove the coincidence of
the local $L$-factors at every prime.
To show this, we use the recent results of Harris-Lan-Taylor-Thorne and Scholze
on the construction of Galois representations,
and Greni\'e's results to compare three-dimensional $2$-adic Galois representations.
We also prove the local-global compatibility at $p = 2$,
including the case $p = \ell$.
\end{abstract}

\section{Introduction}

In this paper, we prove that
the motive of rank three over $\Q$ with $\Q(\sqrt{-1})$-coefficients constructed by van Geemen-Top in \cite{vanGeemenTop:Invent}
is associated with a cuspidal non-selfdual automorphic representation of $\GL_3(\A_{\Q})$ of level $\Gamma_0(128)$.
We shall prove the coincidence of the local $L$-factors at every $p \geq 3$.
We shall also prove the local-global compatibility at $p = 2$,
including the case $p = \ell$.

To the knowledge of the authors, our results give the first example of a non-selfdual motive of rank three over $\Q$ associated with a cuspidal non-selfdual automorphic representation of $\GL_n(\A_{\Q})$ ($n \geq 3$)
whose associated $\ell$-adic representations
are not ``geometrically polarizable,'' more precisely, not odd in the sense of \cite[Definition 2.3]{JohanssonThorne}. (As explained in \emph{ibid.}, Galois representations appearing in the cohomology of Shimura varieties are expected to be odd.)

\subsection{The results of van Geemen-Top}

Before giving the statements of our results,
we briefly recall the results of van Geemen-Top.
(See \cite{vanGeemenTop:Invent} for details.)

We choose a square root $\sqrt{-1} \in \C$ of $-1$,
and fix an isomorphism
$\iota_{\ell} \colon \overline{\Q}_{\ell} \cong \C$ for every $\ell$.
By decomposing the transcendental part of the $\ell$-adic cohomology of
a projective smooth surface $S_2$ whose affine model is
\[ t^2 = xy(x^2 - 1)(y^2 - 1)(x^2 - y^2 + 2 xy), \]
van Geemen-Top constructed a three-dimensional $\ell$-adic representation
\[ \rho_{\mathrm{vGT},\ell} \colon \Gal(\overline{\Q}/\Q) \to \GL_3(\Q(\sqrt{-1})_{v_{\ell}}). \]
Here $v_{\ell}$ is the finite place of $\Q(\sqrt{-1})$ above $\ell$
given by the fixed isomorphism $\iota_{\ell}$.

A remarkable property of the compatible system $\{ \rho_{\mathrm{vGT},\ell} \}_{\ell}$
is that it is associated with a non-selfdual automorphic
representation of $\GL_3(\A_\Q)$.
It was discovered experimentally by van Geemen-Top.

Let $\Gamma_0(128) \subset \SL_3(\Z)$ be the subgroup consisting of matrices $(a_{ij})_{1 \leq i,j \leq 3}$ with $a_{21} \equiv a_{31} \equiv 0 \pmod{128}$.
In \cite{vanGeemenTop:Invent},
van Geemen-Top found a group cohomology class 
\[ u \in H^3(\Gamma_0(128),\C), \]
which is an eigenclass for Hecke operators.
It can be shown that $u$ is cuspidal.
(See Section \ref{Remark:Cuspidality}.)

Let $\pi$ be the cuspidal automorphic representation of $\GL_3(\A_\Q)$ associated with $u$.
We decompose $\pi$ into the restricted tensor product
$\pi = \bigotimes'_{p \leq \infty} \pi_p$.
By calculating the action of Hecke operators explicitly,
they confirmed the equality of local $L$-factors
\begin{equation}
\label{EqualityLfactors}
\det(1 - p^{-s} \rho_{\mathrm{vGT},\ell}(\Frob_p)) = L(s,\pi_p)
\end{equation}
for every $p$ with $3 \leq p \leq 67$ and $p \neq \ell$.
Here $\Frob_p$ denotes the geometric Frobenius element.
(See Section \ref{Subsection:GeometricFrobenius}.)
Later, this result was extended to $p \leq 173$ by
van Geemen-van der Kallen-Top-Verberkmoes
\cite{vanGeemenvanderKallenTopVerberkmoes}.

\begin{rem}
The cuspidal automorphic representation $\pi$ is written as $\pi_u$ in \cite[Section 2.4]{vanGeemenTop:Invent}.
It is not unitary because
the central character of $\pi$ is $\lvert \cdot \rvert^{-3}$.
(Here $\lvert \cdot \rvert \colon \A_{\Q}^{\times} \to \C^{\times}$
is the id\`ele norm.)
The functional equation relates $L(s,\pi)$ and $L(3-s,\pi^{\vee})$,
where $\pi^{\vee}$ is the contragredient representation of $\pi$.
\end{rem}

\subsection{The main results of this paper}

In this paper, we shall prove
that the equality (\ref{EqualityLfactors}) holds for every $p \geq 3$.
Moreover, we shall prove that
the compatible system $\{ \rho_{\mathrm{vGT},\ell} \}_{\ell}$
is associated with the automorphic representation $\pi$ and they satisfy the local-global compatibility at $p=2$, including the case $p = \ell$.

Here is the main theorem of this paper.

\begin{thm}
\label{MainTheorem}
Let $\ell, p$ be prime numbers.
\begin{enumerate}
\item Assume $p \neq \ell$.
The $\ell$-adic representation $\rho_{\mathrm{vGT},\ell}$
is unramified at $p$ if and only if $p \neq 2$.
Moreover, we have
\[
\mathrm{WD}(\rho_{\mathrm{vGT},\ell}|_{W_{\Q_p}})^{\text{\rm F-ss}}
\cong \iota_{\ell}^{-1} \mathrm{rec}_{\Q_p}(\pi_p).
\]

\item Assume $p = \ell$.
The $p$-adic representation $\rho_{\mathrm{vGT},p}$ is de Rham at $p$ with Hodge-Tate weights $0$, $1$, $2$, each with multiplicity one.
It is crystalline at $p$ if and only if $p \neq 2$.
Moreover, we have
\[
\mathrm{WD}(\rho_{\mathrm{vGT},p}|_{\Gal(\overline{\Q}_p/\Q_p)})^{\text{\rm F-ss}}
\cong \iota_p^{-1} \mathrm{rec}_{\Q_p}(\pi_p).
\]
\end{enumerate}
\end{thm}

Here  $\mathrm{rec}_{\Q_p}$ denotes the local Langlands correspondence
for $\GL_3$ over $\Q_p$ as normalized in \cite{HarrisTaylor}.
Thus
$\iota_{\ell}^{-1} \mathrm{rec}_{\Q_p}(\pi_p)$
is the Weil-Deligne representation over $\overline{\Q}_{\ell}$ associated with the irreducible admissible representation $\pi_p$ of $\GL_3(\Q_p)$.

In Theorem \ref{MainTheorem} (1),
the left hand side is the Frobenius semisimplification of
the Weil-Deligne representation
associated with the restriction of $\rho_{\mathrm{vGT},\ell}$ to the Weil group $W_{\Q_p}$.
(See \cite[Section 1]{TaylorYoshida} for details.)
If $p \neq 2$, the isomorphism in Theorem \ref{MainTheorem} (1)
is nothing but a restatement of the equality of local $L$-factors (\ref{EqualityLfactors}).

Theorem \ref{MainTheorem} (2)
concerns the $p$-adic properties of the compatible system.
We follow the convention that the $p$-adic cyclotomic character has Hodge-Tate weight $-1$.
The left hand side of Theorem \ref{MainTheorem} (2)
is the Frobenius semisimplification of the Weil-Deligne representation
constructed by Fontaine's $\widehat{D}_{\mathrm{pst}}$-functor.
(See \cite[Section 1]{TaylorYoshida}.)
If $p \neq 2$, Theorem \ref{MainTheorem} (2)
means the local $L$-factor $L(s,\pi_p)$ is equal to the local $L$-factor of the $p$-adic representation $\rho_{\mathrm{vGT},p}$ at $p$
given by the crystalline Frobenius
via the crystalline comparison theorem.

Theorem \ref{MainTheorem} asserts the equality of
the $L$-functions coming from motives and automorphic representations.
As usual, it has several standard consequences
both on the motivic side and the automorphic side.

For example, on the motivic side, we have the following results.

\begin{cor}
\label{Corollary:HasseWeil}
The Hasse-Weil $L$-function of the surface $S_2$ studied by van Geemen-Top
has analytic continuation to the complex plane,
and it satisfies a functional equation.
\end{cor}

In fact, the first and the third Betti numbers of $S_2$ are zero.
(See \cite[Section 3.1]{vanGeemenTop:Invent}.)
The transcendental part of $H^2$ is the direct sum of
$\rho_{\mathrm{vGT},\ell}$ and the contragredient representation of
$\rho_{\mathrm{vGT},\ell}(2)$, where $(2)$ denotes the Tate twist.
(See Remark \ref{Remark:ConjugateGaloisRepresentation}.)

On the automorphic side, we have the following results.

\begin{cor}
The  cuspidal non-selfdual automorphic representation $\pi$
of $\GL_3(\A_\Q)$ of level $\Gamma_0(128)$
found by van Geemen-Top
satisfies the Ramanujan-Petersson conjecture,
i.e.\ the local component $\pi_p$ is
(essentially) tempered for every $p$, including $p = 2$.
\end{cor}

\begin{rem}
To the knowledge of the authors,
our results give the first example of
a cuspidal non-selfdual automorphic representation of
$\GL_n(\A_{\Q})$ ($n \geq 3$)
such that the Ramanujan-Petersson conjecture is proved for it
and the associated $\ell$-adic representations are
not odd in the sense of \cite[Definition 2.3]{JohanssonThorne}.
(See Remark \ref{Remark:ZariskiDensity}.)
\end{rem}

\subsection{Cuspidality of the eigenclass $u$}
\label{Remark:Cuspidality}

The cuspidality of the eigenclass $u$ was presumably well-known. But it was not stated in \cite{vanGeemenTop:Invent}.
Here we briefly explain how to associate an automorphic
representation $\pi$ to the eigenclass $u$ without assuming
the cuspidality of $u$.
Then we shall show that $u$ and $\pi$ are cuspidal.

Let $\widehat{\Z}$ be the profinite completion of $\Z$.
Let $U_0(128) \subset \GL_3(\widehat{\Z})$ be the open compact subgroup consisting of matrices $(a_{ij})_{1 \leq i,j \leq 3}$ with $a_{21} \equiv a_{31} \equiv 0 \pmod{128}$.
Let $\R_{>0}$ be the multiplicative group of positive real numbers.
We consider it as a subgroup of $\GL_3(\R)$ via the diagonal embedding.
The symmetric space of $\mathrm{SL}_3(\R)$ is denoted by
\[ X_{\infty} := \mathrm{SL}_3(\R)/\mathrm{SO}_3(\R). \]
Then we have an isomorphism
\begin{equation}
\label{EqualitySymmetricSpace}
\GL_3(\Q) \R_{>0} \backslash \GL_3(\A_{\Q}) / U_0(128) \mathrm{SO}_3(\R) \cong \Gamma_0(128) \backslash X_{\infty}.
\end{equation}

Recall that the eigenclass $u \in H^3(\Gamma_0(128), \C)$ is cuspidal
(i.e.\ it belongs to the subspace $H^3_{!}(\Gamma_0(128), \C)$; see \cite[Section 2]{AshGraysonGreen})
if and only if the system of Hecke eigenvalues given by $u$
is equal to the system of Hecke eigenvalues given by a cuspidal automorphic
representation of $\GL_3(\A_\Q)$.
(See \cite[Theorem 4.18]{HarderRaghuram} for example.)

If we assume $u$ is cuspidal, we can consider it as a cuspidal automorphic form on $X_{\infty}$
as in \cite[Section 2]{AshGraysonGreen}.
Then, by (\ref{EqualitySymmetricSpace}), we can consider it as a cuspidal automorphic form on $\GL_3(\A_{\Q})$.
We take the cuspidal automorphic representation of $\GL_3(\A_{\Q})$ generated by it, and twist it by $\lvert \det \rvert^{-1}$.
However, if we do not assume $u$ is cuspidal, there does not seem an easy way to construct an automorphic form on $X_{\infty}$.

We shall overcome this difficulty using Franke's theorem  as follows.
The group cohomology $H^3(\Gamma_0(128), \C)$ is identified with the cohomology of the left hand side of (\ref{EqualitySymmetricSpace}) with coefficients in the constant local system $\C$.
Franke proved the existence of a decomposition of the following form
\begin{align*}
&\qquad H^3(\GL_3(\Q) \R_{>0} \backslash \GL_3(\A_{\Q}) / U_0(128) \mathrm{SO}_3(\R),\, \C) \\
&\cong \bigoplus_{\{ P \}} \mathrm{Ext}^3_{(\mathrm{Lie}(\SL_3(\R)), \mathrm{SO}_3(\R))}(\C,\, S_{\infty}(\GL_3(\Q) \R_{>0} \backslash \GL_3(\A_{\Q}))_{\{ P \}}^{U_0(128)}).
\end{align*}
(See \cite[Section 7.6]{Franke} for details.)
The eigenclass $u$ gives an element of one of
the direct summands corresponding to a parabolic subgroup $P \subset \GL_3$.
As explained in Step 2 of \cite[p.~262]{Franke},
it corresponds to a regular algebraic automorphic representation
of $\GL_3(\A_{\Q})$.
We denote it by $\pi'$. Then we put
\[ \pi := \pi' \otimes \lvert \det \rvert^{-1}. \]
Since the central character of $\pi'$ is trivial,
the central character of $\pi$ is $\lvert \cdot \rvert^{-3}$.

\begin{rem}
The reason why we take the $(-1)$-twist is that
the $L$-factors calculated by van Geemen-Top in \cite{vanGeemenTop:Invent} coincide with the $L$-factors of $\pi$.
In \cite{vanGeemenTop:Invent},
van Geemen-Top used Ash-Grayson-Green's algorithm to calculate Hecke eigenvalues.
For details on the algorithm, see \cite[Section 4]{AshGraysonGreen}.
\end{rem}

We shall show the automorphic representation $\pi$ is cuspidal.
By the same argument as in Step 2 of \cite[p.~262]{Franke},
there exists a partition $3 = \sum_{i=1}^{k} n_i$
and a regular $L$-algebraic cuspidal automorphic representation
$\pi'_i$ of $\GL_{n_i}(\A_{\Q})$ for each $i$ such that
\[ L(s,\pi'_p) = \prod_{i=1}^{k} L(s,\pi'_{i,p}) \]
for every $p \neq 2$.
(For the definition of $L$-algebraic representations,
see \cite[Definition 5.10]{BuzzardGee}.)
If $\pi$ were not cuspidal,
at least one of $n_1,\ldots,n_k$ is equal to $1$.
We may assume $n_1 = 1$.
Then $\pi'_1 \colon \A_{\Q}^{\times} \to \C^{\times}$
is an algebraic Hecke character,
and its local $L$-factor $L(s,\pi'_{1,p})$ is of the form
$(1 - \chi(p) p^{m-s})^{-1}$
for an integer $m$ and a Dirichlet character $\chi$.
Since $L(s,\pi) = L(s-1,\pi')$,
the local $L$-factor $L(s,\pi_p)$ is divisible by
$(1 - \chi(p) p^{m+1-s})^{-1}$.
However, an explicit calculation of the local $L$-factor at $p=5$ shows that $L(s,\pi_5)$ does not satisfy this condition. (See \cite[Remark 3.12 (2)]{vanGeemenTop:Invent}.)
Therefore, the automorphic representation $\pi$ is cuspidal.
Consequently, by \cite[Theorem 4.18]{HarderRaghuram},
the eigenclass $u$ is cuspidal.

\begin{rem}
\label{Remark:UnitaryTwist}
The automorphic representation $\pi' = \pi \otimes \lvert \det \rvert$ is unitary because $\pi'$ is cuspidal and the central character of $\pi'$ is trivial.
\end{rem}

\begin{rem}
In principle, it should be possible to confirm the cuspidality of $u$ by a finite amount of calculation.
However, it seems that such calculation was not yet performed.
According to the last paragraph of \cite[Section 2.6]{vanGeemenvanderKallenTopVerberkmoes},
there does not exist a practical algorithm to check the cuspidality of
an eigenclass in the group cohomology unless the level is prime.
\end{rem}

\subsection{Outline of the proof of Theorem \ref{MainTheorem}}

We first use recent results
on the construction of $\ell$-adic representations of $\Gal(\overline{\Q}/\Q)$
associated with regular algebraic cuspidal automorphic representations of $\GL_3(\A_{\Q})$.
Such results were proved by Harris-Lan-Taylor-Thorne
\cite[Theorem A]{HLTT}.
A little later, they were also proved by Scholze by a different method \cite[Theorem 1.0.4]{Scholze:Torsion}\footnote{At the time of writing of this paper,
the main results of Scholze in \cite{Scholze:Torsion}
are still conditional.
But, since we work only  over $\Q$, the results we need in this paper were unconditionally proved in \cite{Scholze:Torsion}. (See \cite[Remark 5.4.6]{Scholze:Torsion}.)}.
By the Chebotarev density theorem
(see \cite[Chapter I, Section 2.2]{Serre:l-adic}),
the $\ell$-adic representations constructed by Harris-Lan-Taylor-Thorne and Scholze are equivalent.

Let us take $\ell = 2$.
By the results of Harris-Lan-Taylor-Thorne and Scholze,
we have a three-dimensional semisimple 2-adic representation
\[ \rho_{\pi,2} \colon \Gal(\overline{\Q}/\Q) \to \GL_3(\overline{\Q}_2) \]
such that it is unramified outside $2$ and $\infty$,
and it satisfies the equality
\[
\det(1 - p^{-s} \rho_{\pi,2}(\Frob_p))^{-1}
= L(s,\pi_p)
\]
for every $p \geq 3$.

We shall show the equivalence of $\rho_{\mathrm{vGT},2}$ and $\rho_{\pi,2}$
using Greni\'e's results \cite{Grenie}.
To use his results, we need to show the image of $\rho_{\pi,2}$ is contained in a subgroup conjugate to  $\GL_3(\Q(\sqrt{-1})_{v_2})$.
Here $v_2$ is the finite place of $\Q(\sqrt{-1})$ above $2$.
It does not seem to follow from the construction of $\rho_{\pi,2}$
in \cite{HLTT}, \cite{Scholze:Torsion}.
Fortunately, we can prove it using a representation-theoretic result
(see Lemma \ref{Lemma:RepresentationImage}).
Consequently, $\rho_{\mathrm{vGT},2}$, $\rho_{\pi,2}$ are equivalent,
and Theorem \ref{MainTheorem} is proved in the case $p \neq 2$ and $p \neq \ell$.
By the Chebotarev density theorem,
$\rho_{\mathrm{vGT},\ell}$ is equivalent to the semisimple $\ell$-adic representation $\rho_{\pi,\ell}$ associated with $\pi$.

Once Theorem \ref{MainTheorem} is proved in the case $p \neq 2$ and $p \neq \ell$,
it is standard to deduce Theorem \ref{MainTheorem}
in the case $p \neq 2$ and $p = \ell$
by comparing the characteristic polynomials of Frobenius
acting on the $\ell$-adic cohomology and the crystalline cohomology.

If $p = 2$, there is another difficulty:
the local-global compatibility was not fully established in
\cite{HLTT}, \cite{Scholze:Torsion}.
We shall use a rather  indirect argument to prove
Theorem \ref{MainTheorem} in the case $p = 2$.

If $p = 2$ and $p \neq \ell$,
we shall use the recent results of Varma on the local-global compatibility
up to semisimplification \cite{Varma}.
It implies $\rho_{\mathrm{vGT},\ell} \cong \rho_{\pi,\ell}$ satisfies Theorem \ref{MainTheorem} (1),
except possibly for the equality of the monodromy operators.
We shall show the equality of the monodromy operators
by the method of Taylor-Yoshida \cite{TaylorYoshida}.
We have already shown that $\rho_{\pi,\ell}$ is equivalent to
$\rho_{\mathrm{vGT},\ell}$.
Up to a quadratic twist,
$\rho_{\mathrm{vGT},\ell}$ is embedded in
the cohomology of a projective smooth surface;
for such $\ell$-adic representations,
the weight-monodromy conjecture is known to be true.
From this, we see that the Weil-Deligne representation
associated with
$\rho_{\mathrm{vGT},\ell} \cong \rho_{\pi,\ell}$
is pure in the sense of \cite[Section 1, p.~471]{TaylorYoshida}.
It is well-known that purity implies Theorem \ref{MainTheorem} (1).

Finally, the case $p = \ell = 2$ is proved by
a combination of Varma's results and
Ochiai's $\ell$-independence results \cite{Ochiai}.
We will actually need a slight generalization of Ochiai's results for varieties equipped with automorphisms of finite order.
We shall prove it using a trick of unramified twist.
(See Lemma \ref{Lemma:Independence}.)

\begin{rem}
A natural question is whether the methods of this paper can be
applied to other motives conjecturally associated with
non-selfdual (or selfdual) automorphic representations.
Note that van Geemen-Top constructed several motives of rank three over $\Q$;
see
\cite{vanGeemenTop:Invent},
\cite{vanGeemenTop:Compositio},
\cite{vanGeemenTop:LNCS},
\cite{vanGeemenvanderKallenTopVerberkmoes}.
It is an interesting problem to study them by the methods of this paper.
However, it is not clear whether the methods of this paper
can be applied to $\ell$-adic representations
which are ramified at some places outside $2, \infty$.
In this paper,  we crucially use Greni\'e's results to compare $2$-adic representations.
Greni\'e's results depend on rather special properties of a tower of number fields unramified outside $2, \infty$.
(See Theorem \ref{Theorem:Grenie}. See also \cite[Section 4]{Grenie}.)
\end{rem}

\subsection{Outline of this paper}

The outline of this paper is as follows.
In Section \ref{Section:GaloisRepresentations},
we recall basic results on Galois representations.
We also recall Greni\'e's results.
In Section \ref{Section:GaloisRepresentations:vanGeemenTop},
we recall the results of van Geemen-Top
on the construction of the compatible system 
$\{ \rho_{\mathrm{vGT},\ell} \}_{\ell}$.
In Section \ref{Section:HLTTScholze},
we briefly recall the results of Harris-Lan-Taylor-Thorne
and Scholze on the construction of $\ell$-adic representations.
Finally, in Section \ref{Section:ProofMainTheorem},
we prove Theorem \ref{MainTheorem}.

\section{$\ell$-adic representations of $\Gal(\overline{\Q}/\Q)$}
\label{Section:GaloisRepresentations}

In this section, we recall basic results on Galois representations.
We also fix notation used in this paper.

As in Introduction, we choose a square root $\sqrt{-1} \in \C$ of $-1$,
and fix an isomorphism
$\iota_{\ell} \colon \overline{\Q}_{\ell} \cong \C$
for every $\ell$.
For an element $a \in \C$, its complex conjugate is denoted by $a^c$.

\subsection{$\ell$-adic representations}

Let $\ell$ be a prime number.
An \textit{$\ell$-adic representation} of $\Gal(\overline{\Q}/\Q)$
means a continuous homomorphism 
\[ \rho \colon \Gal(\overline{\Q}/\Q) \to \GL_n(E) \]
for some $n \geq 1$.
Here $E$ is a finite extension of $\Q_{\ell}$ or $E = \overline{\Q}_{\ell}$.
For a subfield $E' \subset E$,
we say that $\rho$ is \textit{defined over $E'$} if the image of $\rho$ is contained in a subgroup conjugate to $\GL_n(E')$.
Any $\ell$-adic representation over $\overline{\Q}_{\ell}$
is defined over a finite extension of $\Q_{\ell}$.
(See \cite[Lemme 2.2.1.1]{BreuilMezard}, \cite[Corollary 5]{Dickinson}.)

Let $\rho$ be an $\ell$-adic representation defined over $E$.
Then the trace $\Tr \rho(\sigma)$ is in $E$ for every $\sigma \in \Gal(\overline{\Q}/\Q)$.
But, it is not always the case that $\rho$ is defined over
the field generated
by the traces of elements in the image of $\rho$.
At least for a semisimple representation,
there is a simple criterion for it to be defined over
the field generated by the traces.

The following result is presumably well-known.
It was used in the proof of
\cite[Proposition 3.2.5]{ChenevierHarris:PartII}.
In \cite[Proposition 7]{Chin},
Chin proved a similar result for absolutely irreducible representations.

\begin{lem}
\label{Lemma:RepresentationImage}
Let $G$ be a group, and $K$ a field of characteristic $0$.
Let $\rho \colon G \to \GL_n(\overline{K})$ be
a semisimple representation defined over $\overline{K}$.
Let $\rho \cong \rho_1 \oplus \cdots \oplus \rho_r$
be an irreducible decomposition of $\rho$.
Assume that the following conditions are satisfied:
\begin{enumerate}
\item $\rho$ is defined over a finite extension of $K$. 
\item $\Tr \rho(g) \in K$ for every $g \in G$.
\item There is an element $g_0 \in G$
such that the characteristic polynomial of $\rho(g_0)$
has $n$ distinct roots in $K$.
\end{enumerate}
Then each $\rho_i$ ($1 \leq i \leq r$) is defined over $K$.
In particular, $\rho$ is defined over $K$.
\end{lem}

\begin{proof}
We shall reduce the proof of this lemma to the case studied by Chin.
We decompose $\overline{K}^{\oplus n}$ into irreducible representations:
\[ \overline{K}^{\oplus n} = V_1 \oplus \cdots \oplus V_r, \]
where $V_i$ corresponds to $\rho_i$.
For each $i$, we take a non-zero eigenvector $v_i \in V_i$ of $\rho(g_0)$ with
eigenvalue $\alpha_i$,
i.e.\ we have $\rho(g_0) v_i = \alpha_i v_i$.
By assumption, $\alpha_1,\ldots,\alpha_r$ are distinct elements of $K$.
Take an element $\sigma \in \Gal(\overline{K}/K)$,
and consider the map
$\sigma \rho \colon G \to \GL_n(\overline{K})$.
Since the traces of elements of $G$ are contained in $K$,
$\sigma \rho$ is equivalent to $\rho$ as a representation of $G$.
We have the irreducible decomposition with respect to the action
of $G$ via $\sigma \rho$:
\[ \overline{K}^{\oplus n} = \sigma(V_1) \oplus \cdots \oplus \sigma(V_r). \]
For each $i$, we have
\[ (\sigma \rho(g_0)) (\sigma v_i) = \sigma (\rho(g_0) v_i)
= \sigma(\alpha_i  v_i) = \alpha_i (\sigma v_i). \]
(Recall that $\sigma \alpha_i = \alpha_i$ because $\alpha_i$ is an element of $K$.)
Thus, $\sigma(V_i)$ contains an eigenvector of $\rho(g_0)$
with respect to the eigenvalue $\alpha_i$.
This implies $V_i$ is equivalent to $\sigma(V_i)$ as a representation of $G$.
Hence $\Tr \rho_i(g) \in K$ for every $g \in G$.
By \cite[Proposition 7]{Chin}, we conclude that $\rho_i$ is defined over $K$.
\end{proof}

\subsection{Geometric Frobenius elements}
\label{Subsection:GeometricFrobenius}

We take a prime number $p$.
The decomposition group at $p$ is denoted by
$\Gal(\overline{\Q}_p/\Q_p) \hookrightarrow \Gal(\overline{\Q}/\Q)$.
It is well-defined up to conjugation.

We have a natural surjection
$\Gal(\overline{\Q}_p/\Q_p) \to \Gal(\overline{\F}_p/\F_p)$.
Its kernel, denoted by $I_p$, is called the \textit{inertia group} at $p$.
Let $\Frob_p \in \Gal(\overline{\F}_p/\F_p)$ be
the \textit{geometric Frobenius element}
defined by $x \mapsto x^{1/p} \ (x \in \overline{\F}_p)$.
Let $W_{\Q_p} \subset \Gal(\overline{\Q}_p/\Q_p)$
be the \textit{Weil group} of $\Q_p$,
which is defined by the inverse image of
$\{ \Frob_p^n \mid n \in \Z \}$.

We say an $\ell$-adic representation $\rho$ is \textit{unramified} at $p$ if the image $\rho(I_p)$ is trivial.
Otherwise, we say it is \textit{ramified} at $p$.
(We never consider the notion of ``ramification at $\infty$'' in this paper.)
We choose a lift $\widetilde{\Frob}_p \in \Gal(\overline{\Q}_p/\Q_p)$,
and consider it as an element of $\Gal(\overline{\Q}/\Q)$ via
an embedding
$\Gal(\overline{\Q}_p/\Q_p) \hookrightarrow \Gal(\overline{\Q}/\Q)$.
If $\rho$ is unramified at $p$,
the image $\rho(\widetilde{\Frob}_p)$ is well-defined up to conjugation.
We denote it by $\rho(\Frob_p)$, by a slight abuse of notation.
The characteristic polynomial of $\rho(\Frob_p)$ is well-defined.

\subsection{Greni\'e's results}

It is well-known that for each $n \geq 1$, a finite set of prime numbers $S$, a prime number $\ell$, and a finite extension $E$ of $\Q_{\ell}$,
there exists a finite set of prime numbers $T_{n,S,E}$ with
$S \cap T_{n,S,E} = \emptyset$ such that for any $n$-dimensional semisimple $\ell$-adic representations
\[ \rho_1, \rho_2 \colon \Gal(\overline{\Q}/\Q) \to \GL_n(E) \]
defined over $E$ which are unramified outside $S \cup \{ \infty \}$,
if the characteristic polynomials of $\rho_1(\Frob_p), \rho_2(\Frob_p)$
are equal for every $p \in T_{n,S,E}$,
then $\rho_1, \rho_2$ are equivalent.
(See \cite[Th\'eor\`eme 3.1]{Deligne:TateShafarevich} for example.)
Such a set $T_{n,S,E}$ can be effectively calculated in principle.
However, it can be a huge set in general.

In \cite{Grenie}, Greni\'e gave a method to obtain a set $T_{n,S,E}$ satisfying the above condition.
His result can be applied to $\ell$-adic representations satisfying some restrictive conditions only.
(For the precise statement, see \cite[Theorem 3]{Grenie}.)
When it can be applied, it sometimes gives an efficient criterion to test the equivalence of $\ell$-adic representations.
The following result is proved in \cite[Section 4]{Grenie}.

\begin{thm}[Greni\'e]
\label{Theorem:Grenie}
Let $n \in \{ 3,4 \}$.
Let $E$ be a finite totally ramified extension of $\Q_2$.
Let
\[ \rho_1, \rho_2 \colon \Gal(\overline{\Q}/\Q) \to \GL_n(E) \]
be $n$-dimensional semisimple $2$-adic representations
defined over $E$ such that they are unramified outside $2$ and $\infty$,
and they satisfy
\[
\det(T - \rho_1(\Frob_p)) = \det(T - \rho_2(\Frob_p))
\]
for $p \in \{ 5,7,11,17,23,31 \}$.
Then $\rho_1, \rho_2$ are equivalent.
\end{thm}

\begin{proof}
See \cite[Section 4]{Grenie}.
\end{proof}

\section{Three-dimensional $\ell$-adic representations
constructed by van Geemen-Top}
\label{Section:GaloisRepresentations:vanGeemenTop}

\subsection{van Geemen-Top's construction}
\label{Subsection:vanGeemenTopConstruction}

In \cite{vanGeemenTop:Invent}, van Geemen-Top studied
the cohomology of the projective smooth surface over $\Q$
with affine model
\[ t^2 = xy(x^2 - 1)(y^2 - 1)(x^2 - y^2 + 2 xy). \]
This surface was denoted by $S_2$ in \cite[Section 3.2]{vanGeemenTop:Invent}.
It admits an automorphism of order $4$, given by
\[ \phi \colon S_2 \to S_2,\quad (x,y,t) \mapsto (y,-x,t). \]

Let $\ell$ be a prime number.
Let $v_{\ell}$ be the finite place of $\Q(\sqrt{-1})$ above $\ell$
given by the fixed isomorphism
$\iota_{\ell} \colon \overline{\Q}_{\ell} \cong \C$.
The transcendental part of the $\ell$-adic cohomology is denoted by
\[ T_{\Q_{\ell}} \subset H^2_{\text{\'et}}((S_2)_{\overline{\Q}}, \Q_{\ell}). \]
It is defined by the orthogonal complement of the image of the Chern class map
\[
\mathrm{cl}^1_{\text{\'et}} \colon \NS((S_2)_{\overline{\Q}}) \otimes_{\Z} \Q_{\ell}(-1)
\to H^2_{\text{\'et}}((S_2)_{\overline{\Q}}, \Q_{\ell})
\]
with respect to the cup product
\[
\cup \colon H^2_{\text{\'et}}((S_2)_{\overline{\Q}}, \Q_{\ell})
\times H^2_{\text{\'et}}((S_2)_{\overline{\Q}}, \Q_{\ell})
\to H^4_{\text{\'et}}((S_2)_{\overline{\Q}}, \Q_{\ell}) \cong \Q_{\ell}(-2).
\]
In \cite[Section 3.3]{vanGeemenTop:Invent},
van Geemen-Top proved $\dim_{\Q_{\ell}} T_{\Q_{\ell}} = 6$.

Let
$\phi^{\ast} \colon T_{\Q_{\ell}} \to T_{\Q_{\ell}}$
be the pullback by the automorphism $\phi$.
We have $\phi^{\ast} \circ \phi^{\ast} = -1$ on $T_{\Q_{\ell}}$. (See \cite[Section 3.4]{vanGeemenTop:Invent}.)
We decompose the $\Q(\sqrt{-1})_{v_{\ell}}$-vector space
$T_{\Q_{\ell}} \otimes_{\Q_{\ell}} \Q(\sqrt{-1})_{v_{\ell}}$
into two eigenspaces
\[
T_{\Q_{\ell}} \otimes_{\Q_{\ell}} \Q(\sqrt{-1})_{v_{\ell}}
= V_{+} \oplus V_{-},
\]
where $V_{+}$ (resp.\ $V_{-}$) is the eigenspace of $\phi^{\ast}$ with respect to the eigenvalue $\sqrt{-1}$ (resp.\ $-\sqrt{-1}$).
Since $\phi^{\ast}$ commutes with the action of $\Gal(\overline{\Q}/\Q)$,
we have a natural action of $\Gal(\overline{\Q}/\Q)$ on $V_{+}$ and $V_{-}$.

We consider the action of $\Gal(\overline{\Q}/\Q)$
on $V_{+}$ twisted by the quadratic character
\[ \chi_{-2} \colon \Gal(\overline{\Q}/\Q) \to \{ \pm 1 \} \]
corresponding to $\Q(\sqrt{-2})/\Q$.
The $\Q(\sqrt{-1})_{v_{\ell}}$-vector space $V_{+}$
with the twisted action by $\chi_{-2}$ is denoted by
$V_{+} \otimes \chi_{-2}$.
It gives a three-dimensional $\ell$-adic representation
defined over $\Q(\sqrt{-1})_{v_{\ell}}$.
We denote it by
\[
 \rho_{\mathrm{vGT},\ell} \colon \Gal(\overline{\Q}/\Q) \to
   \GL(V_{+} \otimes \chi_{-2}) \cong \GL_3(\Q(\sqrt{-1})_{v_{\ell}}).
\]

\begin{rem}
\label{Remark:ZariskiDensity}
The $\ell$-adic representation
$\rho_{\mathrm{vGT},\ell}$
has Zariski dense image in $\GL_3(\overline{\Q}_{\ell})$.
(See \cite[p.~7, Section 2]{JohanssonThorne}.)
In particular, it is not odd in the sense of \cite[Definition 2.3]{JohanssonThorne}.
\end{rem}

\begin{rem}
\label{Remark:ConjugateGaloisRepresentation}
We can also consider the action of $\Gal(\overline{\Q}/\Q)$
on $V_{-}$ twisted by $\chi_{-2}$.
We denote it by
\[
 \rho_{\mathrm{vGT},\ell}^{c} \colon \Gal(\overline{\Q}/\Q) \to
   \GL(V_{-} \otimes \chi_{-2}) \cong \GL_3(\Q(\sqrt{-1})_{v_{\ell}}).
\]
The cup product on
$H^2_{\text{\'et}}((S_2)_{\overline{\Q}}, \Q_{\ell})$
induces a $\Gal(\overline{\Q}/\Q)$-invariant perfect pairing
\[
\cup \colon T_{\Q_{\ell}} \times T_{\Q_{\ell}} \to \Q_{\ell}(-2).
\]
It is easy to see that the restrictions of $\cup$ to $V_{+},V_{-}$
are trivial.
Hence $\rho_{\mathrm{vGT},\ell}^{c}$ is equivalent to
the contragredient representation
of $\rho_{\mathrm{vGT},\ell}(2)$,
where $(2)$ denotes the Tate twist.
The construction of $\rho_{\mathrm{vGT},\ell}$ depends on
the choice of $\sqrt{-1} \in \C$ and $\iota_{\ell} \colon \overline{\Q}_{\ell} \cong \C$.
If we replace $\sqrt{-1} \in \C$ by its negative,
the $\ell$-adic representation $\rho_{\mathrm{vGT},\ell}$
is replaced by $\rho_{\mathrm{vGT},\ell}^{c}$.
\end{rem}

\subsection{Properties of $\rho_{\mathrm{vGT},\ell}$ at $p$ ($p \neq \ell$)}

\begin{prop}
\label{Proposition:vanGeemenTopGalois}
\begin{enumerate}
\item $\rho_{\mathrm{vGT},\ell}$ is absolutely irreducible.
It is unramified outside $2$ and $\infty$.

\item For any $p, \ell$ with $p \neq \ell$ and $p \neq 2$,
the characteristic polynomial of the geometric Frobenius element at $p$ is written as
\[
\iota_{\ell}(\det(T - \rho_{\mathrm{vGT},\ell}(\Frob_p) )) = T^3 - b_p T^2 + p (b_p)^c T - p^3
\]
for some $b_p \in \Q(\sqrt{-1})$.
(Here $(b_p)^c$ is the complex conjugate of $b_p$.)

\item The element $b_p \in \Q(\sqrt{-1})$ in (2)
is independent of $\ell \neq p$ and $\iota_{\ell}$.

\item For any $\ell \neq 2$ and any $\sigma \in W_{\Q_2}$,
the characteristic polynomial
\[ \iota_{\ell}(\det(T - \rho_{\mathrm{vGT},\ell}(\sigma) )) \]
has coefficients in $\Q(\sqrt{-1})$.
It is independent of $\ell \neq 2$ and $\iota_{\ell}$.
\end{enumerate}
\end{prop}

\begin{proof}
(1), (2) are proved in \cite[Section 3.7]{vanGeemenTop:Invent}
except for the absolute irreducibility.
The absolute irreducibility of $\rho_{\mathrm{vGT},\ell}$
is proved in \cite[Remark 3.12]{vanGeemenTop:Invent}
and the last paragraph of \cite[Section 4.1]{Grenie}.

(3) As explained in \cite[Section 3.8]{vanGeemenTop:Invent},
we have
\[
b_p = \iota_{\ell}(\Tr \rho_{\mathrm{vGT},\ell}(\Frob_p)) =
\frac{1}{2} \iota_{\ell}(\Tr(\Frob_p; T_{\Q_{\ell}}))
-
\frac{\sqrt{-1}}{2} \iota_{\ell}(\Tr(\Frob_p \circ \phi^{\ast}; T_{\Q_{\ell}})).
\]
It is enough to show that
$\Tr(\Frob_p; T_{\Q_{\ell}})$,
$\Tr(\Frob_p \circ \phi^{\ast}; T_{\Q_{\ell}})$
are rational numbers which are independent of $\ell \neq p$.
This follows from
\cite[Proposition 3.6]{vanGeemenTop:Invent}
and
\cite[Proposition 3.9]{vanGeemenTop:Invent}.
(In fact,
$\Tr(\Frob_p; T_{\Q_{\ell}})$,
$\Tr(\Frob_p \circ \phi^{\ast}; T_{\Q_{\ell}})$
are integers.)

(4)
By the same reason as (3), it is enough to show that
$\Tr(\sigma; T_{\Q_{\ell}})$,
$\Tr(\sigma \circ \phi^{\ast}; T_{\Q_{\ell}})$
are rational numbers which are independent of $\ell \neq 2$.
To show this, it is enough to show that
$\Tr(\sigma; H^{2}_{\text{\'et}}((S_2)_{\overline{\Q}_2}, \Q_{\ell}))$,
$\Tr(\sigma \circ \phi^{\ast}; H^{2}_{\text{\'et}}((S_2)_{\overline{\Q}_2}, \Q_{\ell}))$
are rational numbers which are independent of $\ell \neq 2$.
It follows from Saito's $\ell$-independence results.
(See \cite[Theorem 0.1]{Saito:l-independence}.)
\end{proof}

\begin{rem}
For any $\ell \neq 2$,
the $\ell$-adic representation
$\rho_{\mathrm{vGT},\ell}$ is ramified at $2$.
(See Remark \ref{Remark:Ramifiedat2}.)
\end{rem}

\begin{rem}
To prove Proposition \ref{Proposition:vanGeemenTopGalois},
we do not need the full strength of Saito's results.
In fact, we can prove Proposition \ref{Proposition:vanGeemenTopGalois}
using Ochiai's results and a trick of unramified twist.
(See Lemma \ref{Lemma:Independence} below.)
\end{rem}

\subsection{Properties of $\rho_{\mathrm{vGT},p}$ at $p$}

Here we consider the case $p = \ell$.
We shall study the properties of the $p$-adic representation $\rho_{\mathrm{vGT},p}$ at $p$.

The following is a slight generalization of Ochiai's $\ell$-independence results \cite{Ochiai}.

\begin{lem}
\label{Lemma:Independence}
Let $p$ be a prime number, and $K$ a finite extension of $\Q_p$.
Let $X$ be a proper smooth variety over $K$,
and $f \in \Aut(X)$ an automorphism of finite order defined over $K$.
Let $\sigma \in W_{K}$ be an element in the Weil group of $K$.

\begin{enumerate}
\item For $\ell \neq p$, the alternating sum of the traces
\[
\sum_{i = 0}^{2 \dim X} (-1)^i \Tr(\sigma \circ f^{\ast}; H^{i}_{\text{\'et}}(X_{\overline{K}}, \Q_{\ell}))
\]
are rational numbers which are independent of
$\ell (\neq p)$.

\item Moreover, it is equal to
\[
\sum_{i = 0}^{2 \dim X} (-1)^i \Tr(\sigma \circ f^{\ast}; \mathrm{WD}(H^{i}_{\text{\'et}}(X_{\overline{K}}, \Q_p))).
\]
Here $\mathrm{WD}(H^{i}_{\text{\'et}}(X_{\overline{K}}, \Q_p))$
is the Weil-Deligne representation associated with
the $p$-adic representation
$H^{i}_{\text{\'et}}(X_{\overline{K}}, \Q_p)$
of $\Gal(\overline{K}/K)$.
(See \cite[Section 1]{TaylorYoshida} for details.)
\end{enumerate}
\end{lem}

\begin{proof}
If $f$ is the identity, these assertions follow from
Ochiai's results \cite[Theorem B, Theorem D]{Ochiai}.

If $f$ is not the identity, (1) follows from Saito's results \cite[Theorem 0.1]{Saito:l-independence}.
If we could prove Saito's results in the $p$-adic case, (2) could follow.
However, since such results are not currently available,
we shall use a trick of unramified twist to deduce
both (1) and (2) from Ochiai's results.
(The same trick was used by Deligne-Lusztig for varieties over a finite field.
See \cite[p.~119]{DeligneLusztig}.)

Let $q$ be the number of elements in the residue field of $K$.
For $\sigma \in W_{K}$,
its image in $\Gal(\overline{\F}_q/\F_q)$
is written as $\Frob_q^n$ for some integer $n$.
We put $n(\sigma) := n$.
(Here $\Frob_q$ is the geometric Frobenius element
defined by $x \mapsto x^{1/q} \ (x \in \overline{\F}_q)$.)

For $\ell \neq p$ and $\sigma \in W_K$, we put
\begin{align*}
t(\sigma,\ell) &:= \sum_{i = 0}^{2 \dim X} (-1)^i \Tr(\sigma \circ f^{\ast}; H^{i}_{\text{\'et}}(X_{\overline{K}}, \Q_{\ell})), \\
t(\sigma,p) &:= \sum_{i = 0}^{2 \dim X} (-1)^i \Tr(\sigma \circ f^{\ast}; \mathrm{WD}(H^{i}_{\text{\'et}}(X_{\overline{K}}, \Q_p))).
\end{align*}

To prove (1) and (2),
it is enough to prove $t(\sigma,\ell) = t(\sigma,p)$
for every $\ell \neq p$ and every $\sigma \in W_K$.
In the following argument, we fix $\ell \neq p$.

We claim that if $t(\sigma,\ell) = t(\sigma,p)$
holds for every $\sigma \in W_K$ with $n(\sigma) \geq 1$,
then $t(\sigma,\ell) = t(\sigma,p)$
holds for every $\sigma \in W_K$.
It is presumably well-known.
(See \cite[Lemma 1]{Saito:modularform} for example.)
We briefly give a proof. 
We may replace $H^{i}_{\text{\'et}}(X_{\overline{K}}, \Q_{\ell})$
by the Weil-Deligne representation $\mathrm{WD}(H^{i}_{\text{\'et}}(X_{\overline{K}}, \Q_{\ell}))$.
Then the action of the inertia group on
$\mathrm{WD}(H^{i}_{\text{\'et}}(X_{\overline{K}}, \Q_{\ell}))$
(resp.\ $\mathrm{WD}(H^{i}_{\text{\'et}}(X_{\overline{K}}, \Q_p))$)
factors through a finite quotient.
Thus, there is a positive integer $r \geq 1$
such that, for every $i$, the actions of $\tau^r$ on
$\mathrm{WD}(H^{i}_{\text{\'et}}(X_{\overline{K}}, \Q_{\ell}))$
and
$\mathrm{WD}(H^{i}_{\text{\'et}}(X_{\overline{K}}, \Q_p))$
commute with the actions of every element of $W_K$.
Take an element $\tau \in W_{K}$ with $n(\tau) = 1$.
Let $\sigma \in W_K$ be an element.
It is an exercise in linear algebra that
if $t(\tau^{rs} \circ \sigma, \ell) = t(\tau^{rs} \circ \sigma, p)$
holds for every sufficiently large integer $s \gg 1$,
then it also holds for every integer $s$, including $s = 0$.
From this, the claim follows.

Further, replacing $K$ by a finite unramified extension of it,
it is enough to prove $t(\sigma,\ell) = t(\sigma,p)$
for every $\sigma \in W_K$ with $n(\sigma) = 1$.

Let $\sigma \in W_K$ be an element with $n(\sigma) = 1$.
Let $d$ be the order of the automorphism $f$.
Let $K_d/K$ be the unramified extension of degree $d$.
Taking the twist of $X$ with respect to $f$ and $\Gal(K_d/K)$,
we have a proper smooth variety $X_f$ over $K$.
It is isomorphic to $X$ over $K_d$.
Since $\sigma$ is a lift of $\Frob_q$,
the action of
$\mathrm{id}_{X_f} \otimes \sigma^{\ast}$ on $X_f \otimes_{K} K_d$
is identified with the action of
$f \otimes \sigma^{\ast}$ on $X \otimes_{K} K_d$.
Hence we have
\[
   \Tr(\sigma; H^{i}_{\text{\'et}}((X_f)_{\overline{K}}, \Q_{\ell}))
 = \Tr(\sigma \circ f^{\ast}; H^{i}_{\text{\'et}}(X_{\overline{K}}, \Q_{\ell})).
\]
Similarly, we also have
\[
   \Tr(\sigma; \mathrm{WD}(H^{i}_{\text{\'et}}((X_f)_{\overline{K}}, \Q_p)))
 = \Tr(\sigma \circ f^{\ast}; \mathrm{WD}(H^{i}_{\text{\'et}}(X_{\overline{K}}, \Q_p))).
\]
Therefore, it is enough to prove
\[
\sum_{i = 0}^{2 \dim X} (-1)^i \Tr(\sigma; H^{i}_{\text{\'et}}((X_f)_{\overline{K}}, \Q_{\ell}))
= \sum_{i = 0}^{2 \dim X} (-1)^i \Tr(\sigma; \mathrm{WD}(H^{i}_{\text{\'et}}((X_f)_{\overline{K}}, \Q_p))).
\]
It follows from \cite[Theorem D]{Ochiai}.
\end{proof}

We shall use Lemma \ref{Lemma:Independence}
to compare the properties of the $\ell$-adic representation
$\rho_{\mathrm{vGT},\ell}$
with the properties of the $p$-adic representation
$\rho_{\mathrm{vGT},p}$.

The $p$-adic representation $\rho_{\mathrm{vGT},p}$ is de Rham at $p$
because it is, up to a quadratic twist,
embedded as a direct summand into the $p$-adic cohomology of the projective smooth surface $S_2$.
Let
$\mathrm{WD}(\rho_{\mathrm{vGT},p}|_{\Gal(\overline{\Q}_p/\Q_p)})$
be the Weil-Deligne representation associated with it.
(See \cite[Section 1]{TaylorYoshida}.)

\begin{prop}
\label{Proposition:vanGeemenTopGalois:l=p}
\begin{enumerate}
\item Assume $p \neq 2$.
The $p$-adic representation $\rho_{\mathrm{vGT},p}$ is crystalline at $p$,
and we have
\[
\iota_p(\det(T - \Frob_p; \mathrm{WD}(\rho_{\mathrm{vGT},p}|_{\Gal(\overline{\Q}_p/\Q_p)}))) = T^3 - b_p T^2 + p (b_p)^c T - p^3.
\]
Here $b_p \in \Q(\sqrt{-1})$ is the same element as
Proposition \ref{Proposition:vanGeemenTopGalois} (2).

\item For any $\ell \neq 2$ and any $\sigma \in W_{\Q_2}$,
we have the equality of the characteristic polynomials
\[ \iota_{2}(\det(T - \sigma; \mathrm{WD}(\rho_{\mathrm{vGT},2}|_{\Gal(\overline{\Q}_2/\Q_2)}) ))
= \iota_{\ell}(\det(T - \rho_{\mathrm{vGT},\ell}(\sigma) )). \]
\end{enumerate}
\end{prop}

\begin{proof}
If $p \neq 2$, since $S_a$ has good reduction outside $2$,
the $p$-adic representation $\rho_{\mathrm{vGT},p}$
is crystalline at $p$ by the crystalline comparison theorem.
(See \cite[Theorem 0.2]{Tsuji}.)

The characteristic polynomial of $\Frob_p$ in (1)
(resp.\ $\sigma$ in (2)) can be calculated
from the traces of powers of $\Frob_p$ (resp.\ $\sigma$).
Hence it is enough to prove
\[
\iota_p(\Tr(\sigma; \mathrm{WD}(\rho_{\mathrm{vGT},p}|_{\Gal(\overline{\Q}_p/\Q_p)})))
= \iota_{\ell}(\Tr \rho_{\mathrm{vGT},\ell}(\sigma))
\]
for any $\sigma \in W_{\Q_p}$ and any $p \neq \ell$ (including $p = 2$).

The equalities of the traces are known on
the subspace of $H^2$ generated by divisors.
As in the proof of Proposition \ref{Proposition:vanGeemenTopGalois} (4),
it is enough to show the equalities of the traces of
$\sigma$, $\sigma \circ \phi^{\ast}$ on $H^2$.
Since the first and the third Betti numbers of $S_2$
are zero (see \cite[Section 3.1]{vanGeemenTop:Invent}),
it is enough to show the following equalities of the alternating sums of the traces for $p \neq \ell$
\begin{align*}
\sum_{i=0}^{4} (-1)^i \Tr(\sigma; \mathrm{WD}(H^i_{\text{\'et}}((S_2)_{\overline{\Q}_p}, \Q_p)))
&=
\sum_{i=0}^{4} (-1)^i \Tr(\sigma; H^i_{\text{\'et}}((S_2)_{\overline{\Q}_p}, \Q_{\ell})), \\
\sum_{i=0}^{4} (-1)^i \Tr(\sigma \circ \phi^{\ast}; \mathrm{WD}(H^i_{\text{\'et}}((S_2)_{\overline{\Q}_p}, \Q_p)))
&=
\sum_{i=0}^{4} (-1)^i \Tr(\sigma \circ \phi^{\ast}; H^i_{\text{\'et}}((S_2)_{\overline{\Q}_p}, \Q_{\ell})).
\end{align*}
These equalities follow from Lemma \ref{Lemma:Independence}.
\end{proof}

\begin{rem}
The $p$-adic representation $\rho_{\mathrm{vGT},p}$
is crystalline at $p$ if and only if $p \neq 2$.
(See Remark \ref{Remark:Ramifiedat2:p=2}.)
\end{rem}

\section{$\ell$-adic representations constructed by
Harris-Lan-Taylor-Thorne and Scholze}
\label{Section:HLTTScholze}

Let $\pi$ be the cuspidal automorphic representation of $\GL_3(\A_\Q)$
of level $\Gamma_0(128)$
associated with the eigenclass $u$.
(See Section \ref{Remark:Cuspidality}.)
Since $\pi$ is cohomological,
we can apply the results of Harris-Lan-Taylor-Thorne and Scholze on the construction of $\ell$-adic representations.
(See \cite[Theorem A]{HLTT}, \cite[Theorem 1.0.4]{Scholze:Torsion}.)
Thus, for every $\ell$,
we have a three-dimensional semisimple $\ell$-adic representation
\[ \rho_{\pi,\ell} \colon \Gal(\overline{\Q}/\Q) \to \GL_3(\overline{\Q}_{\ell}) \]
such that it is unramified outside $2, \ell$ and $\infty$,
and it satisfies the equality
\[
\det(1 - p^{-s} \rho_{\pi,\ell}(\Frob_p))^{-1} = L(s,\pi_p)
\]
for every $p \notin \{ 2, \ell \}$.

The $\ell$-adic representation
$\rho_{\pi,\ell}$ is defined over a finite extension of $\Q_{\ell}$.
However, 
since the construction of $\rho_{\pi,\ell}$ in \cite{HLTT}, \cite{Scholze:Torsion}
was achieved by a combination of several techniques in Galois representations
such as congruences and patching lemmas,
it is not clear how to control the image of $\rho_{\pi,\ell}$
from its construction.

In practice, 
we can use Lemma \ref{Lemma:RepresentationImage}
to control the image.

\begin{lem}
\label{Lemma:ImageAssociated2-adicRepresentation}
Assume $\ell = 2$.
The $2$-adic representation $\rho_{\pi,2}$ is defined over $\Q(\sqrt{-1})_{v_2}$.
Here $v_2$ is the finite place of $\Q(\sqrt{-1})$ above $2$.
\end{lem}

\begin{proof}
By Lemma \ref{Lemma:RepresentationImage},
it is enough to find an element in the image of
$\Gal(\overline{\Q}/\Q)$
whose characteristic polynomial
has three distinct roots in $\Q(\sqrt{-1})_{v_2}$.
For example, $\rho_{\pi,2}(\Frob_{3})$ and $\rho_{\pi,2}(\Frob_{11})$ satisfy
this condition.
In fact, it can be checked that each of the characteristic polynomials
\begin{align*}
\det(T - \rho_{\pi,2}(\Frob_3)) &= (T - 3 \sqrt{-1})(T^2 - (1 - \sqrt{-1}) T - 9 \sqrt{-1})), \\
\det(T - \rho_{\pi,2}(\Frob_{11})) &= T^3 + (7+10\sqrt{-1}) T^2 - 11 \cdot (7+10\sqrt{-1}) T - 11^3
\end{align*}
has three distinct roots in $\Q(\sqrt{-1})_{v_2}$.
\end{proof}

\begin{rem}
For a prime number $p$ with $3 \leq p \leq 67$,
the characteristic polynomial of $\rho_{\pi,2}(\Frob_p)$
has three distinct roots in $\Q(\sqrt{-1})_{v_2}$
if and only if $p \in \{ 3, 11, 19, 47, 61 \}$.
(See \cite[Section 2.5]{vanGeemenTop:Invent}.)
\end{rem}

\begin{rem}
As a corollary of Theorem \ref{MainTheorem},
it can be shown that $\rho_{\pi,\ell}$ is defined over
$\Q(\sqrt{-1})_{v_{\ell}}$ for every $\ell$.
\end{rem}

\begin{rem}
The $\ell$-adic representation $\rho_{\pi,\ell}$
constructed in \cite{HLTT}, \cite{Scholze:Torsion} was expected to
satisfy the following properties:
\begin{enumerate}
\item $\rho_{\pi,\ell}$ is irreducible.
\item $\rho_{\pi,\ell}$ is de Rham at $\ell$. It is crystalline at $\ell$ if and only if $\ell \neq 2$.
\end{enumerate}
These properties were not studied in \cite{HLTT}, \cite{Scholze:Torsion}.
We do not assume them in this paper.
In fact, we can prove them by comparing
$\rho_{\pi,\ell}$ with $\rho_{\mathrm{vGT},\ell}$.
(See Theorem \ref{MainTheorem}.)
\end{rem}

\section{Proof of Theorem \ref{MainTheorem}}
\label{Section:ProofMainTheorem}

\subsection{The case $p \neq 2$ and $p \neq \ell$}

We first consider the $2$-adic representations.

Let $\rho_{\mathrm{vGT},2}$
be the $2$-adic representation constructed by van Geemen-Top
in \cite{vanGeemenTop:Invent}.
(See Section \ref{Subsection:vanGeemenTopConstruction}.)
It is defined over $\Q(\sqrt{-1})_{v_2}$ by construction.
By Proposition \ref{Proposition:vanGeemenTopGalois} (1), it is absolutely irreducible.
It satisfies the equality of local $L$-factors
\[ \det(1 - p^{-s} \rho_{\mathrm{vGT},2}(\Frob_p))^{-1} = L(s,\pi_p) \]
for every $p$ with $3 \leq p \leq 67$.

On the other hand, let
$\rho_{\pi,2}$ be the semisimple $2$-adic representation associated with $\pi$
constructed by Harris-Lan-Taylor-Thorne and Scholze.
(See Section \ref{Section:HLTTScholze}.)
It satisfies
\[
\det(1 - p^{-s} \rho_{\pi,2}(\Frob_p))^{-1} = L(s,\pi_p)
\]
for every $p \geq 3$.
Moreover, by Lemma \ref{Lemma:ImageAssociated2-adicRepresentation},
$\rho_{\pi,2}$ is defined over $\Q(\sqrt{-1})_{v_2}$.

The characteristic polynomials of
$\rho_{\mathrm{vGT},2}(\Frob_p), \rho_{\pi,2}(\Frob_p)$
are equal for every $p$ with $3 \leq p \leq 67$.
By Theorem \ref{Theorem:Grenie}, we conclude that $\rho_{\mathrm{vGT},2}, \rho_{\pi,2}$ are equivalent.
Theorem \ref{MainTheorem} is proved in the case $p \neq 2$ and $\ell = 2$.

From this result, it is easy to deduce the case $p \neq 2$ and $p \neq \ell$ of Theorem \ref{MainTheorem}
by the following corollary.

\begin{cor}
\label{Corollary:EquivalenceGalois}
For any $\ell$,
the $\ell$-adic representations $\rho_{\mathrm{vGT},\ell}$ and $\rho_{\pi,\ell}$
are equivalent.
\end{cor}

\begin{proof}
By Proposition \ref{Proposition:vanGeemenTopGalois} (2),(3),
the characteristic polynomials of 
$\rho_{\mathrm{vGT},\ell}(\Frob_p)$, $\rho_{\mathrm{vGT},2}(\Frob_p)$
are equal for every $p \notin \{ 2,\ell \}$.
On the other hand, since the equality
\[
\det(1 - p^{-s} \rho_{\pi,\ell}(\Frob_p))^{-1} = L(s,\pi_p)
\]
holds for every $p \notin \{ 2, \ell \}$,
the characteristic polynomials of 
$\rho_{\pi,\ell}(\Frob_p)$, $\rho_{\pi,2}(\Frob_p)$
are equal for every $p \notin \{ 2,\ell \}$.
Thus, the characteristic polynomials of
$\rho_{\mathrm{vGT},\ell}(\Frob_p)$, $\rho_{\pi,\ell}(\Frob_p)$
are equal for every $p \notin \{ 2,\ell \}$.
By the Chebotarev density theorem
(see \cite[Chapter I, Section 2.2]{Serre:l-adic}),
$\rho_{\mathrm{vGT},\ell}$, $\rho_{\pi,\ell}$
have equivalent semisimplifications.
By Proposition \ref{Proposition:vanGeemenTopGalois} (1),
$\rho_{\mathrm{vGT},\ell}$ is absolutely irreducible.
Hence $\rho_{\pi,\ell}$ is also absolutely irreducible, and $\rho_{\mathrm{vGT},\ell}$, $\rho_{\pi,\ell}$
are equivalent.
\end{proof}

The proof of Theorem \ref{MainTheorem}
in the case $p \neq 2$ and $p \neq \ell$ is complete.

\subsection{The case $p \neq 2$ and $p = \ell$}

Next, we shall prove Theorem \ref{MainTheorem}
in the case $p \neq 2$ and $p = \ell$.

The assertion on the Hodge-Tate weights follows from the calculation of the Hodge numbers of the surface $S_2$ (see \cite[Section 3.2]{vanGeemenTop:Invent})
and the comparison theorem between de Rham cohomology and $p$-adic \'etale cohomology (see \cite[Theorem A1]{Tsuji:Survey}).

The assertion for local $L$-factors follows from Corollary \ref{Corollary:EquivalenceGalois} and 
Proposition \ref{Proposition:vanGeemenTopGalois:l=p} (1).

\subsection{The case $p = 2$ and $p \neq \ell$}

Here we shall prove Theorem \ref{MainTheorem}
in the case $p = 2$ and $p \neq \ell$.
For this purpose, we need the recent results of Varma.

In \cite{Varma}, Varma proved the local-global compatibility,
up to semisimplification.
(See \cite[Theorem ($1^{ss}$)]{Varma} for the precise statement.)
By Varma's results, for any $\ell \neq 2$,
$\mathrm{WD}(\rho_{\pi,\ell}|_{W_{\Q_2}})^{\text{\rm F-ss}}$
and
$\iota_{\ell}^{-1} \mathrm{rec}_{\Q_2}(\pi_2)$
are equivalent as representations of the Weil group $W_{\Q_2}$.

We need to prove the coincidence, up to equivalence, of the monodromy operators.
We shall do it by the method of Taylor-Yoshida \cite{TaylorYoshida}.
By Corollary \ref{Corollary:EquivalenceGalois},
$\rho_{\mathrm{vGT},\ell}$, $\rho_{\pi,\ell}$ are equivalent as $\ell$-adic representations of $\Gal(\overline{\Q}/\Q)$.
Thus, we have an equivalence of Weil-Deligne representations
\[
\mathrm{WD}(\rho_{\mathrm{vGT},\ell}|_{W_{\Q_2}})^{\text{\rm F-ss}}
\cong
\mathrm{WD}(\rho_{\pi,\ell}|_{W_{\Q_2}})^{\text{\rm F-ss}}.
\]

The $\ell$-adic representation $\rho_{\mathrm{vGT},\ell}$ is,
up to a quadratic twist,
embedded as a direct summand into the $\ell$-adic cohomology of $S_2$.
(See Section \ref{Subsection:vanGeemenTopConstruction}.)
Since the weight-monodromy conjecture is known to be true for surfaces
(see \cite[Lemma 3.9]{Saito:l-independence}),
we see that
$\mathrm{WD}(\rho_{\pi,\ell}|_{W_{\Q_2}})^{\text{\rm F-ss}}$
is \textit{pure} in the sense of \cite[Section 1, p.~471]{TaylorYoshida}.
(Precisely, the above argument only gives
$\mathrm{WD}(\rho_{\pi,\ell}|_{W_{\Q_2}})^{\text{\rm F-ss}}$
is $\iota_{\ell}$-pure.
Applying the whole arguments to every choice of $\iota_{\ell}$,
we see that it is pure.)

The rest of the proof is exactly the same as in \cite{TaylorYoshida}.
The purity of
$\mathrm{WD}(\rho_{\pi,\ell}|_{W_{\Q_2}})^{\text{\rm F-ss}}$
implies it is \textit{mixed} as a representation of $W_{\Q_2}$.
For every $\sigma \in W_{\Q_2}$,
every eigenvalue $\alpha$ of $\rho_{\pi,\ell}(\sigma)$ is an algebraic number
such that $\lvert \iota(\alpha) \rvert^2$
is an integer power of $2$ for any embedding
$\iota \colon \Q(\alpha) \hookrightarrow \C$.

Recall that $\pi' = \pi \otimes \lvert \det \rvert$
is a unitary cuspidal automorphic representation of $\GL_3(\A_{\Q})$.
(See Remark \ref{Remark:UnitaryTwist}.)
As a representation of $W_{\Q_2}$,
$\iota_{\ell}^{-1} \mathrm{rec}_{\Q_2}(\pi'_2)$ is equivalent to
the semisimplification of $\rho_{\pi,\ell}(1)$,
where $(1)$ denotes the Tate twist.
Therefore, for every $\sigma \in W_{\Q_2}$,
every eigenvalue $\alpha'$ of $\iota_{\ell}^{-1} \mathrm{rec}_{\Q_2}(\pi'_2)$
is an algebraic number
such that $\lvert \iota(\alpha') \rvert^2$
is an integer power of $2$ for any embedding
$\iota \colon \Q(\alpha') \hookrightarrow \C$.
By \cite[Corollary VII.2.18]{HarrisTaylor}, $\pi'_2$ is absolutely tempered.
By \cite[Lemma 1.4 (3)]{TaylorYoshida},
the Weil-Deligne representation $\iota_{\ell}^{-1} \mathrm{rec}_{\Q_2}(\pi'_2)$ is pure.
Hence its Tate twist
$\iota_{\ell}^{-1} \mathrm{rec}_{\Q_2}(\pi_2) = (\iota_{\ell}^{-1} \mathrm{rec}_{\Q_2}(\pi'_2))(-1)$ 
is also pure.

Since both
$\mathrm{WD}(\rho_{\pi,\ell}|_{W_{\Q_2}})^{\text{\rm F-ss}}$
and
$\iota_{\ell}^{-1} \mathrm{rec}_{\Q_2}(\pi_2)$
are pure,
the monodromy operators on them coincide,
up to equivalence, by \cite[Lemma 1.4 (4)]{TaylorYoshida}.

The proof of Theorem \ref{MainTheorem}
in the case $p = 2$ and $p \neq \ell$ is complete.

\begin{rem}
In this paper, we follow the convention
that ``tempered'' means the twist by a character is
a unitary tempered representation.
(See also \cite[p.~470]{TaylorYoshida}.)
\end{rem}

\begin{rem}
There is a typo in the statement of \cite[Lemma 1.4 (3)]{TaylorYoshida}; we need the condition that the absolute values of
the central characters of $\sigma \pi$ are the same for all $\sigma \colon \Omega \hookrightarrow \C$.
In our situation, since the central character of $\pi'_2$ (resp.\ $\pi_2$) is trivial (resp.\ $\lvert \cdot \rvert^{-3}$),
we can apply \cite[Lemma 1.4 (3)]{TaylorYoshida} to $\pi'_2$ (resp.\ $\pi_2$).
\end{rem}

\begin{rem}
\label{Remark:Ramifiedat2}
The irreducible admissible representation $\pi_2$ of $\GL_3(\Q_2)$ is not unramified,
i.e.\ it does not have a non-zero vector fixed by $\GL_3(\Z_2)$.
The local-global compatibility proved as above implies
the $\ell$-adic representation
$\rho_{\pi,\ell} \cong \rho_{\mathrm{vGT},\ell}$
is ramified at $2$ for every $\ell \neq 2$.
\end{rem}

\subsection{The case $p = \ell = 2$}

Finally, we shall prove the case $p = \ell = 2$ of
Theorem \ref{MainTheorem}.
This is the most delicate case because
the properties of
the associated $2$-adic representation $\rho_{\pi,2}$ at $2$ were not known
by the results of Harris-Lan-Taylor-Thorne and Scholze.
We need to switch to $\rho_{\mathrm{vGT},2}$ to obtain necessary results.

By the same argument as before,
in the case $p = \ell = 2$,
we need to prove the following:
\begin{enumerate}
\item For every $\sigma \in W_{\Q_2}$, we have the equality of the traces
\[
\iota_2(\Tr(\sigma; \mathrm{WD}(\rho_{\pi,2}|_{\Gal(\overline{\Q}_2/\Q_2)})))
=
\Tr(\sigma; \mathrm{rec}_{\Q_2}(\pi_2)).
\]

\item The Weil-Deligne representation $\mathrm{WD}(\rho_{\pi,2}|_{\Gal(\overline{\Q}_2/\Q_2)})$
is pure in the sense of \cite[Section 1, p.~471]{TaylorYoshida}.
\end{enumerate}

We shall prove (1).
Take an auxiliary prime number $\ell' \neq 2$.
For every $\sigma \in W_{\Q_2}$,
by the results of Varma (see \cite[Theorem ($1^{ss}$)]{Varma}),
we have
\[
\iota_{\ell'}(\Tr \rho_{\pi,\ell'}(\sigma))
= \Tr(\sigma; \mathrm{rec}_{\Q_2}(\pi_2)).
\]
Hence the assertion follows from
Corollary \ref{Corollary:EquivalenceGalois} and 
Proposition \ref{Proposition:vanGeemenTopGalois:l=p} (2).

We shall prove (2).
By the same argument as in the case $p = 2$ and $p \neq \ell$,
the $2$-adic representation $\rho_{\pi,2}|_{\Gal(\overline{\Q}_2/\Q_2)} \cong
\rho_{\mathrm{vGT},2}|_{\Gal(\overline{\Q}_2/\Q_2)}$
is, up to a quadratic twist,
embedded as a direct summand into the $2$-adic cohomology of $S_2$.
The purity of
$\mathrm{WD}(\rho_{\pi,2}|_{\Gal(\overline{\Q}_2/\Q_2)})$
is a consequence of the weight-monodromy conjecture for the 
$2$-adic \'etale cohomology of proper smooth surfaces
over $2$-adic fields.
This follows from de Jong's alteration, 
the comparison theorem in the semi-stable case \cite{Tsuji},
and the weight spectral sequence of Mokrane \cite{Mokrane}.

This completes the proof of the case $p = \ell = 2$ of
Theorem \ref{MainTheorem}.

The proof of Theorem \ref{MainTheorem} is complete.

\begin{rem}
\label{Remark:Ramifiedat2:p=2}
Since $\pi_2$ is not unramified,
the local-global compatibility in the case $p=\ell=2$ proved as above implies the $2$-adic representation
$\rho_{\pi,2} \cong \rho_{\mathrm{vGT},2}$
is not crystalline at $2$.
\end{rem}

\subsection*{Acknowledgements}

The work of the first author was supported by JSPS KAKENHI Grant Number 20674001 and 26800013.
The work of the third author was supported by JSPS KAKENHI Grant Number 15H03605.

\end{document}